\documentclass[review]{elsarticle}

\usepackage{graphicx}
\usepackage{amsmath,amsthm,amsfonts,amscd,amssymb,comment,eucal,latexsym,mathrsfs}
\usepackage{stmaryrd}
\usepackage[all]{xy}

\usepackage{epsfig}

\usepackage[all]{xy}
\xyoption{poly}
\usepackage{fancyhdr}
\usepackage{wrapfig}
\usepackage{epsfig}
\usepackage{titletoc,tocloft}
\usepackage{hyperref}
\usepackage{enumitem}
\usepackage{tikz}
\usetikzlibrary{positioning,chains,fit,shapes,calc}

\theoremstyle{plain}
\newtheorem{thm}{Theorem}[section]
\newtheorem{prop}[thm]{Proposition}
\newtheorem{lem}[thm]{Lemma}

\newtheorem*{thm*}{Theorem}

\theoremstyle{definition}

\theoremstyle{remark}

\advance \topmargin by -\headheight
\advance \topmargin by -\headsep
\evensidemargin \oddsidemargin
\usepackage{tikz}

\definecolor{cof}{RGB}{219,144,71}
\definecolor{pur}{RGB}{186,146,162}
\definecolor{greeo}{RGB}{91,173,69}
\definecolor{greet}{RGB}{52,111,72}

                 \def\R{{\mathbb{R}}}        \def\Z{{\mathbb{Z}}}

  \newcount\colveccount
\newcommand*\colvec[1]{
        \global\colveccount#1
        \begin{bmatrix}
        \colvecnext
}
\def\colvecnext#1{
        #1
        \global\advance\colveccount-1
        \ifnum\colveccount>0
                \\
                \expandafter\colvecnext
        \else
                \end{bmatrix}
        \fi
}

\usepackage{tikz}
\usetikzlibrary{positioning,chains,fit,shapes,calc}
\tikzstyle{vertex}=[circle, draw, inner sep=0pt, minimum size=6pt]

\usepackage{algpseudocode}
\usepackage{algorithmicx}

\usepackage{algorithm}

\newtheorem{definition}{Definition}

\newtheorem*{theorem*}{Theorem}

\tikzstyle{vertex}=[circle, draw, inner sep=0pt, minimum size=6pt]

\usepackage{lineno,hyperref}
\modulolinenumbers[5]

\journal{arxiv}

\begin{document}

\begin{frontmatter}

\title{Matrix representations of transversal matroids}

\author{Austin Alderete}
\corref{mycorrespondingauthor}
\cortext[mycorrespondingauthor]{Corresponding author}
\ead{aaldere@math.utexas.edu}
\address{University of Texas at Austin}

\begin{abstract}
 A transversal matroid $M$ of rank $r$ on $[n]$ can be associated to a family of binary matrices corresponding to different presentations of $M$. We describe those matrices which arise from unique maximal presentations of size $r$- giving a complete characterization when $r=2$-, give an algorithm for producing such a matrix, and show that determining whether a matrix corresponds to a maximal presentation is related to determining the tropical rank of arbitrary matrices. 
\end{abstract}


\end{frontmatter}


\section{Introduction}

Enumerating matroids has been an interest within combinatorics since their conception. Oxley's seminal text provides a wealth of conjectures related to counting. One avenue of attack proceeds by restricting problems to specialized classes of matroids, as was done with lattice-path matroids \cite{Lattice_Path}. Alongside this class of problems is the search for operations that build matroids from existing ones. A fundamental such operation is that of single-element deletion. In 2015, Bonin and de Mier consider the reverse operation, single-element extension. A parallel approach would be the construction of a parameter space for restricted classes of matroids. Then a uniform sampling of the space pushes to uniform sampling of the matroids, enabling probablistic tools to be employed. The motivation behind this paper was to place a subset of matrices in bijection with the class of transversal matroids to allow for arbitrary but uniform construction.

Each transversal matroid of rank $r$ has a unique maximal $r$-presentation. Algorithm 1 provides an easily implementable process for outputting the maximal presentation given an arbitrary $r$-presentation of a transversal matroid. While maximality may be determined from the algorithm itself, Theorem 2.5 gives a necessary and sufficient set of conditions for determining whether a presentation is maximal. Unfortunately, any algorithm capable of checking these conditions is capable of determining tropical rank, a problem known to be NP-hard.

\subsection{Transversal matroids}

In this paper, we restrict our attention to transversal matroids. An introduction to matroid fundamentals can be found in Oxley \cite{Oxley_Matroid}.

Let $E$ be a finite set and let $\mathcal{A} = (A_1,...,A_k)$ be a family of subsets of $E$. A transversal of $\mathcal{A}$ is a tuple $(e_1,...,e_k)$ such that $e_i \in A_i$ and the $e_i$'s are distinct. A transversal of a subsequence $(A_j : j \in J \subseteq [k])$ is a partial transversal of $\mathcal{A}$.

For any such $\mathcal{A}$ and $E$, the set of partial transversals of $\mathcal{A}$ forms the collection of independent sets of a matroid on $E$, denoted $M[\mathcal{A}]$. A transversal matroid $M$ is any matroid $M$ such that $M \cong M[\mathcal{A}]$ for some $\mathcal{A}$ and $\mathcal{A}$ is called a presentation of $M$. We say that a presentation is an $r$-presentation if it is a sequence of length $r$.

Given a presentation $(A_1,...,A_k)$ of a matroid $M$ on $E$, we can associate to it a bipartite graph $\Delta[\mathcal{A}]$ by letting
\[
\mathrm{Vert}(\Delta[\mathcal{A}]) := E \sqcup [k]
\]
and
\[
\mathrm{Edge}(\Delta[\mathcal{A}]) := \{ej : e \in E, j \in [k],~~\mathrm{and}~e \in A_j \}.
\]

We can freely identify the edge set with some subset of $[n] \times [k]$ where $n = |E|$ whenever we are given a labeling $N: E \to [n]$ by $ej \sim (N(e),j)$.

A matching in the bipartite graph $\Delta[\mathcal{A}]$ is a set of edges $ M = \{(e_1,j_1),...,(e_m,j_m)\}$ such that the $e_i$'s are pairwise distinct and the $j_i$'s likewise, so the common underlying vertex set of any two edges in $M$ is empty. The set of all matchings in $\Delta[\mathcal{A}]$ corresponds to the independent sets in $M$.

\subsection{Tropical algebra}

There are several tropical semirings used throughout mathematics \cite{Trop_Geo}. Here, for reasons that will become apparent, we use the max-plus semiring defined as follows.

\begin{definition}(Tropical semiring)
The \textit{max-plus tropical semiring} or simply \textit{tropical semiring} is $(\R \cup -\infty, \oplus, \odot)$ where
\[
a \oplus b = \max\{a,b\} ~~~~;~~~~
a \odot b = a+ b.
\]
\end{definition}

This defines a semiring with additive identity $-\infty$ and multiplicative identity $0$. These operations can be extended to $\R^n$ and $\R^{n\times r}$ by replacing the usual sum and product with their tropical counterparts. E.g.
\[
\begin{bmatrix}
1 & 2 \\
3 & 4
\end{bmatrix}
\odot
\begin{bmatrix}
4 & 3 \\
2 & 1
\end{bmatrix}
=
\begin{bmatrix}
(1 \odot 4) \oplus (2 \odot 2) & 
(1 \odot 3) \oplus (2 \odot 1)\\
(3 \odot 4) \oplus (4 \odot 2) & 
(3 \odot 3) \oplus (4 \odot 1)
\end{bmatrix}
=
\begin{bmatrix}
\max\{5,4\} & \max\{4,3\} \\
\max\{7,6\} & \max\{6,5\}
\end{bmatrix}
=
\begin{bmatrix}
5 & 4\\
7 & 6
\end{bmatrix}
\]

There are two interpretations of the classical rank of a matrix in the tropical setting. We choose to view rank as through the lens of the determinant. We define:

\begin{definition}
The tropical determinant of $M = (m_{ij}) \in \R^{n \times n}$ is
\[
\det_{ \mathrm{Trop}}(M)
=
\bigoplus_{ \sigma \in S_n} \left(
m_{1 \sigma(1)} \odot ... \odot m_{n \sigma(n)}
\right)
\]
\end{definition}

A matrix is said to be \textit{tropically singular} if the maximum of $\det_{Trop}(M)$ is achieved twice. The \textit{tropical rank} $\mathrm{t-rank}(M)$ of a matrix $M$ is then the dimension of its largest non-singular  minor.

As with its classical counterpart, the tropical rank of a matrix is invariant under the taking of a transpose.

We are largely concerned with binary matrices in the tropical setting. The tropical determinant here tells us information about matchings. For future reason, we define a generalized tropical determinant now.

\begin{definition} The \textit{generalized tropical determinant} of a matrix $M = (m_{ij})$ of dimension $n \times r$ with $n \geq r$ is
\[
\det_{\mathrm{Trop}}(M) = 
\bigoplus_{\phi: [r] \hookrightarrow [n]} \left(
m_{\phi(1)1} \odot ... \odot m_{\phi(r)r}
\right).
\]
\end{definition}
By transposition we can always take $r < n$ and in the case in which $r = n$ an injective map $\phi: [r] \hookrightarrow [n]$ is simply an element of $S_n$. Thus, it agrees with the tropical determinant on square matrices.

\section{Maximal representations and algorithm for finding them}

\subsection{Presentations}

Given a transversal matroid of rank $r$, it is always possible to obtain an $r$-presentation for it.

\begin{prop}[Brualdi]
Let $\mathcal{A} = (A_i : i \in J)$ be a presentation of a transversal matroid $M$. If $\phi:  B \to J$ is an injective map from a basis $B$ of $M$ into $J$ such that $\phi(b) = j$ if and only if $b \in A_j$, then $(A_i : i \in \phi(B))$ is also a presentation of $M$. Thus, $M$ has an $r-$presentation $\mathcal{A}'$ where $r = \mathrm{rank}(M)$. Furthermore, if $M$ has no coloops, then all presentations of $M$ have exactly $r(M)$ nonempty sets.
\end{prop}

We can apply this theorem directly to $\Delta[\mathcal{A}]$. Given a perfect matching between $B$ and some subset $K$ of $J$, we can create $\Delta[\mathcal{A}']$ by simply taking the induced subgraph on the vertex set $S \sqcup K$ as demonstrated in Figure 1.

\begin{center}
	\begin{figure}[h]
	\center
		\includegraphics[scale=0.25]{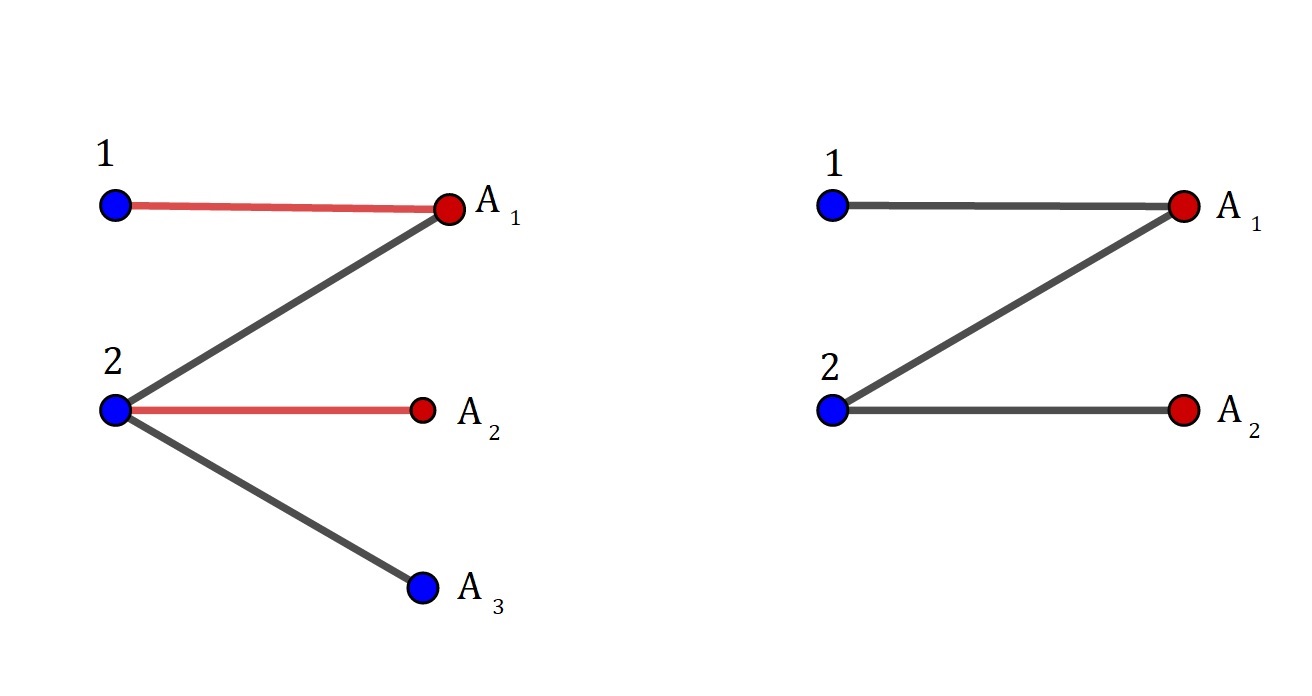}
		\caption{Finding a 2-representation by Prop. 2.1}
	\end{figure}
\end{center}

One can define a partial order on $r$-presentations as follows:
\[
\mathcal{A} = (A_i : i \in [r]) \leq \mathcal{B} = (B_i : i \in [r]) 
~~\iff~~
A_{\sigma(i)} \subseteq B_i~~,~~ \forall i \in [r]
\]
for some $\sigma \in S_r$. Another proposition to be used later allows us to extend a presentation to a greater presentation whenever coloops are present.

\begin{prop}[Bondy-Welsh]
Let $(A_i : i \in [r])$ be a presentation of $M$, a transversal matroid of rank $r$. Then for all $e \in E(M) - A_i$, $(A_1,...,A_i \cup e, A_{i+1},...,A_r)$ is a presentation for $M$ if and only if $e$ is a coloop of $M \backslash A_i$.
\end{prop}

\subsection{Maximal representations}

Under the previous ordering one can talk about minimal and maximal elements. While there are many minimal presentations, the following theorem from Bondy and Mason allows us to speak of the maximal $r$-presentation of a matroid as an equivalence class of presentations under the reordering of their sets\cite{Presentation_Matroid} \cite{Extension_Matroid}\cite{Welsh_Matroid}.

\begin{prop}[Bondy-Mason]
For any transversal matroid $M$ of rank $r$, there exists a maximal presentation $\mathcal{A}$ such that if $\mathcal{B}$ is another $r$-presentation of $M$, then $\mathcal{B} \leq \mathcal{A}$.
\end{prop}

\begin{definition}[Representation of a presentation]
Given a matroid $M$ with presentation $\mathcal{A}$, a matrix $X$ is called a \textit{representation} of $\mathcal{A}$ if $X$ is the biadjacency matrix of $\Delta[\mathcal{A}]$. $X$ is known as a \textit{viable} representation if $r(M) = |\mathcal{A}|$. Similarly, $X$ is called a \textit{maximal} representation if $\mathcal{A}$ is maximal.
\end{definition}

All maximal representations are viable representations. We now present an algorithm whose input is a viable representation of a matroid $M$ and whose output is the maximal representation of $M$. Indexing begins at one for what follows.

\begin{algorithm}
\caption{Algorithm for obtaining a maximal representation}\label{euclid}
\begin{algorithmic}
\State \textbf{Input:} $X$ a viable representation
\State \qquad RowCount(X) := n; ColCount(X) := r;
\State \% We obtain the bipartite graph described by $X$
\State \qquad $G := (V(G) = [n] \sqcup [r],E(G) = \{(a,b): X_{ab} \neq 0\})$.
\State \% For each right-hand vertex (column) we form an induced graph $G_j$ by deleting all adjacent vertices and any edge connected to an adjacent vertex. Note that this deletes far more than just taking the induced graph on $G-\{j\}$.
\State \qquad \textbf{for} $1 \leq j \leq r$:
\State \qquad \qquad $V(G_j) := ([n]-\{v: (v,j) \in E(G)\}) \sqcup ([r]-\{j\})$
\State \qquad \qquad $E(G_j) := \{(a,b) : b \neq j~,~~(a,j) \not \in E(G)\}$
\State \qquad \qquad $M_j := \left\lbrace S : S\subset\mathrm{Left}(V(G_j))~,~S~ \mathrm{has~a~perfect~matching~into~}\mathrm{Right}(V(G_j))~, |S| \geq |S'| \right.$
$\left. ~ \mathrm{for~any~other~perfect~matching}~S' \right\rbrace$.
\State \% $M_j$ is the set of maximal subsets of left vertices of $G_j$ which are perfectly matched into the set of right vertices of $G_j$. We then redefine the graph $G$ itself by adding in all of the edges that were contained in every maximal matching.
\State \qquad \qquad $G: = (V(G) = [n] \sqcup [r],
E(G) := E(G) \bigcup \left\lbrace
(i,j) : i \in \cap_{M \in M_j} M
\right\rbrace
)$
\State \textbf{Output:} The biadjacency matrix for $G$.
\end{algorithmic}
\end{algorithm}

\begin{thm}
Algorithm 1 above outputs a maximal representation for a matroid $M$ when given a viable representation $X$.
\end{thm}

\begin{proof}
Put into the language of matroids, Algorithm 1 is the repeated application of Proposition 2.3, i.e. for each set $A_i$ in the presentation (right-vertices), we take $E(M) -A_i$ (delete adjacent vertices), and look for coloops, that is, elements that appear in every basis of $M \backslash A_i$ (left vertices that appear in every maximal matching). Therefore, at every iteration of the for-loop, a presentation greater than or equal to the initial presentation is produced. It can easily be vertified that if the process is completed at least once, repeating it with any $A_i$ cannot add new elements anywhere in the presentation. Given coloops $e_1,e_2$ of $M \backslash A_i$, it follows that $e_2$ is a coloop of $M \backslash (A_i \cup e_1)$ and so proposition 2.3 applies for sets of coloops and not just single elements at a time. This is used above.

It follows that the algorithm above producese a viable representation and that the associated presentation $\mathcal{B}$ is the largest that can be formed by coloop extensions. Moreover, as Proposition 2.3 is a biconditional statement, there can be no larger presentations formed by extending a set in the family by a single element. But any two presentations with one larger than the other can be linked by intermediate presentations obtained in this way, so $\mathcal{B}$ must be the maximal presentation. By 2.2, $\mathcal{B}$ is unique (up to permutation) and so the order in which the algorithm proceeds (as it indexes $j$ for $1 \leq j \leq r$ does not matter.
\end{proof}

The above algorithm can also be used for maximality testing by simply checking whether the input equals the output. The algorithm is a direct result of the above propositions and is suggested in  \cite{Extension_Matroid}. Classification of transversal matroids by their maximal representation is an appealing notion given that they are matrices and unique up to permutation. There is a difference between labeled and unlabeled, which we note here. If $M$ is an unlabeled matroid with representation $X$, then row and column permutations of $X$ are also representations of $M$. However, if $M$ is a labeled matroid, then only column permutations of its representations are guaranteed to again be representations of $M$. The properties of viability and maximality are preserved likewise.

\subsection{Matrix representations, blocks, and conditions}

A Tutte matrix of a bipartite graph is a matrix $T = (t_{ij})$ where $t_{ij} = 0$ if there is no edge $\{i,j\}$ and an indeterminant otherwise. It is well-known that non-vanishing minors (under the classical determinant) correspond to matchings in the graph \cite{Tutte}.

\begin{definition}[Blocks]
If $X$ is a matrix and $c_i$ is the $i$th column of $X$, then the block $X(c_i)$ is defined by deleting the column $c_i$ from $X$ along with all rows which have a non-zero entry in column $c_i$.
\end{definition}

Example: consider the binary matrix
\[
X = 
\begin{bmatrix}
1 & 1 & 1 \\
0 & 1 & 1 \\
0 & 0 & 1 \\
0 & 0 & 1 \\
1 & 0 & 1 \\
1 & 1 & 1 \\
\end{bmatrix}.
\]

Then,
\[
X(c_1) = 
\begin{bmatrix}
\boxtimes & \boxtimes & \boxtimes \\
\boxtimes & 1 & 1 \\
\boxtimes & 0 & 1 \\
\boxtimes & 0 & 1 \\
\boxtimes & \boxtimes & \boxtimes \\
\boxtimes & \boxtimes & \boxtimes \\
\end{bmatrix}
=
\begin{bmatrix}
1 & 1 \\
0 & 1 \\
0 & 1 \\
\end{bmatrix}.
\]

Note that $X(c_1)$ would describe the matroid on $E(M)- A_1$ under our current system. As such, if $X$ is a Tutte matrix, then the rank of this matroid can be easily checked. However, the condition on coloops requires something more, which appears in Theorem 2.5.

Equivalently, in the binary setting one could consider $X(c_k)$ to be
\[
X(c_k) = (\tilde{x}_{ik} x_{ij})
\]
where $\tilde{x}_{ik} = 1-x_{ik}$. In the above example we would have
\[
X(c_1) = 
\begin{bmatrix}
1(1-1) & 1(1-1) & 1(1-1) \\
0(1-0) & 1(1-0) & 1(1-0) \\
0(1-0) & 0(1-0) & 1(1-0) \\
0(1-0) & 0(1-0) & 1(1-0) \\
1(1-1) & 0(1-1) & 1(1-1) \\
1(1-1) & 1(1-1) & 1(1-1) \\
\end{bmatrix}
=
\begin{bmatrix}
0 & 0 & 0 \\
0 & 1  & 1 \\
0 & 0 & 1 \\
0 & 0 & 1 \\
0 & 0 & 0 \\
0 & 0 & 0 \\
\end{bmatrix}.
\]
This is useful when preserving the indices is of particular importance. In this case, $E(M)- A_1$ inherits the indices from $E(M)$.

We are now ready to present several equivalent necessary and sufficient conditions for a matrix to be a maximal representation. We present two sets of equivalent conditions, one using the Tutte formulation and one using the tropical formulation.

\begin{thm}
Let $X$ be a representation of dimension $n \times r$ and let $T$ be the Tutte matrix associated to the biparite graph described by $X$ as a  biadjacency matrix. Then $X$ is a maximal representation if and only if both (1) and (2) hold:
\begin{enumerate}
\item  rank($T$) = r (i.e. $X$ is viable),
\item for each column $c_i$ of $T$, we have
\[
\mathrm{rank}(T(c_i)) = k \implies \exists (k+1)\mathrm{-many~nonvanishing~} (k \times k)-\mathrm{minors~of~}T(c_i).
\]
\end{enumerate}

Alternatively, let $X$ be a binary representation of dimension $n \times r$. Then $X$ is a maximal representation if and only if both ($1^*$) and ($2^*$) hold:
\begin{itemize}
\item[\textit{1}$^*$.] 
\[
\det_{\mathrm{Trop}}(X) =
\bigoplus_{\phi: [r] \hookrightarrow [n]} \left(
x_{\phi(1)1} \odot ... \odot x_{\phi(r)r} 
\right)
=
r,
\]
\item[\textit{2}$^*$.] for each $k$, the matrix $X(c_k) = (\tilde{x}_{ik} x_{ij})$ has rank $l$ if and only if
\[
\bigoplus_{\phi,\psi: [l] \hookrightarrow [n]} \left(
(\tilde{x}_{ \phi(1)k} x_{ \phi(1) \psi(1)}) \odot ... \odot (\tilde{x}_{ \phi(l)k} x_{ \phi(l) \psi(l)})
\right)
= l
\]
achieves equality at least $(l+1)$-times. 
\end{itemize}
\end{thm}

Note that in $(2^*)$ it is necessary that
\[
\det_{\mathrm{Trop}}(X(c_k)) = l
\]
achieve equality at least $(l+1)$-times, but not sufficient as you could have several different solutions which agreed on their non-zero coordinates in the second formula. We prove the theorem using the Tutte formulation.
\begin{proof}
Let $X$ be our biadjacency matrix, $T$ be our Tutte matrix, $M$ be our matroid, and $\mathcal{A} = (A_i : i \in [r])$ be our presentation described by $X$.

Condition (1) is required for any part of this to make sense (i.e. one cannot talk about the partial order on presentations unless it is a $r(M)$-presentation.)

The requirement for $\mathcal{A}$ to be maximal is that for each $A_i$, the induced matroid on $E(M)-A_i$ must contain no coloops. If a basis of $E(M)-A_i$ has $k$ elements, then there must be at least $k+1$ such bases for there to be no coloops (as a consequence of the basis exchange axiom). $X(c_i)$ is the biadjacency matrix matrix corresponding to $E(M)-A_i$ with presentation $(A_1,...,A_{i-1},A_{i+1},...,A_r)$ and $T(c_i)$ is the Tutte matrix associated to the graph of this presentation. We know that rank($T(c_i)$) $=r(E(M)-A_i)$. Supposing that the rank of $E(M)-A_i$ is $k$, a basis of $E(M)-A_i$ appears as a matching of $k$ left vertices into $k$ right vertices. Then, by the properties of our Tutte matrix, if we restrict to submatrix associated to these vertices, its determinant is non-zero. There must be $k+1$ many distinct bases, and hence $k+1$ nonvanishing $k \times k$ minors.

As all of these properties are biconditional, this is both necessary and sufficient. 
\end{proof}

\section{Classification of rank 2 transversal matroids}

For what follows, fix $r = 2$. All matrices are taken as representative elements of $M_{n \times 2}(\Z / 2 \Z)/(A \sim P_R A P_C)$. Let $X = (x_{i1}~~ x_{i2})$ be one such matrix. Theorem 2.5 states that $X$ is maximal if and only if $X$ has tropical rank $2$ and $X(c_1) = (x_{i2}(1-x_{i1}))$, $X(c_2) = (x_{i1}(1-x_{i2}))$ both achieve their rank $l$ at least $(l+1)$ times.

For the first property, we see that all that is required is that $X$ have a matching. For the second property, we note that $X(c_j)$ is a single column and so can only have rank $1$ or rank $0$. In the rank $1$ case, the number of times that rank is achieved is just the sum of our $X(c_j)$ entries- a property that is exclusive to the $r=2$ case. So we have the following lemma.

\begin{lem}
A representation $X = (x_{ij}) \in \{0,1\}^{n \times 2}$ is maximal if and only if each of the following holds true:
\[
\left(
2 = \bigoplus_{\phi: [2] \to [n]} x_{\phi(1)1} \odot x_{\phi(2)2}
\right)
~~;~~
\left(
\sum_{i}
x_{i2}(1-x_{i1}) \neq 1
\right)
~~;~~
\left(
\sum_{i} x_{i1}(1-x_{i2}) 
\neq 1
\right)
\]
\end{lem}

\begin{thm}
A matrix $X$ of dimension $n \times r$ is a maximal representation of a matroid of rank $r$ if and only if a submatrix of $X$ can be written as $P_R X' P_C$ where $P_R,P_C$ are row and column permutation matrices and $X'$ is an element of
\[
\left\lbrace
\begin{bmatrix}
1 & 1 \\
1 & 1 \\
1 & 1 \\
1 & 1 \\
\vdots & \vdots \\
1 & 1 \\
1 & 1
\end{bmatrix}
,
\begin{bmatrix}
1 & 1 \\
1 & 1 \\
1 & 1 \\
1 & 1 \\
\vdots & \vdots \\
1 & 1 \\
0 & 0
\end{bmatrix}
,
...
,
\begin{bmatrix}
1 & 1 \\
1 & 1 \\
0 & 0 \\
0 & 0 \\
\vdots & \vdots \\
0 & 0 \\
0 & 0
\end{bmatrix}
\right\rbrace
\bigcup
\left\lbrace
\begin{bmatrix}
1 & 1 \\
1 & 1 \\
1 & 1 \\
\vdots & \vdots \\
1 & 1 \\
1 & 0 \\
1 & 0
\end{bmatrix}
,
\begin{bmatrix}
1 & 1 \\
1 & 1 \\
1 & 1 \\
\vdots & \vdots \\
1 & 0 \\
1 & 0 \\
1 & 0
\end{bmatrix}
,
...
,
\begin{bmatrix}
1 & 1 \\
1 & 1 \\
1 & 0 \\
\vdots & \vdots \\
1 & 0 \\
1 & 0 \\
1 & 0
\end{bmatrix}
,
\begin{bmatrix}
1 & 1 \\
1 & 0 \\
1 & 0 \\
\vdots & \vdots \\
1 & 0 \\
1 & 0 \\
1 & 0
\end{bmatrix}
\right\rbrace
\bigcup
\left\lbrace
\begin{matrix}
\begin{bmatrix}
1 & 0
\end{bmatrix}
&  \left. \right\rbrace i :
(n-2) \geq i \geq 2 \\
\begin{bmatrix}
0 & 1 
\end{bmatrix}
& \left. \right\rbrace (n-i)
\end{matrix}
\right\rbrace
\]
\end{thm}

To clarify the notation of the last set, these are matrices with $i$ many rows of the form $[1~~0]$ and $(n-i)$ many rows of the form $[0~~1]$ where $i$ is an integer in $[2,n-2]$. These maximal representations only arise when $n \geq 4$. Moreover, given a maximal representation of dimension $n \times 2$, we can create from it a maximal representation of dimension $(n+k) \times 2$ by adding $k$ rows of the form $[0~~0]$ (corresponding to the addition of loops) or, if $k$ is even, by adding $k$ identical rows either of the form $[1~~0]$ or $[0~~1]$.
\begin{proof}
That these are all maximal representations follows from routine computation. We first note that they all have a matching and then note that $X(c_j)$ has either more than two $1$s or is empty.

So, suppose that we had a matrix $X$, a maximal representation of dimension $n \times 2$. Toward a contradiction, suppose that $X$ has no submatrix equivalent to the forms above. Let $\mathbf{a} = (a_1,a_2,a_3,a_4)$ be a vector such that $a_1$ is the number of rows of the form $[1~~1]$ present in $X$, $a_2$ is the number of rows of the form $[1~~0]$, $a_3$ is the number of rows of the form $[0~~1]$, and $a_4$ is the number of rows of the form $[0~~0]$. By column exchange, we may assume without loss of generality that $a_2 \geq a_3$.

Note that $a_4 \leq (n-2)$, as a maximal representation must at least be viable. If $a_4 \neq 0$, then consider the matrix of dimension $(n-a_4) \times 2$ obtained by removing the all zero rows from $X$. This must be a maximal representation for a matroid of rank $2$ on $[n-a_4]$ as the removal of loops deletes elements without affecting presentations. Then replace $X$ with this matrix and continue. For this reason, we suppose that $a_4 = 0$.

By hypothesis, $a_2 \geq a_3$. If $a_2 = 0$, then our $X$ is of the first form in the union above, so this cannot be. If $a_3 \geq 2$, then $X$ is of our last form. Therefore, $a_3 \leq 1$ (and $1 \leq a_2$).

If $a_3 = 0$, then our matrix is of the middle form, so then we must have that $a_3 = 1$. However, if $a_3 = 1$, then $X(c_1) = [0~~1]$, which has rank $1$ but only one non-zero entry. This contradicts the hypothesis of maximality. Therefore, all maximal representations of dimension $n \times 2$ must have this form.
\end{proof}

\section{Computational complexity}

Given any representation of a transversal matroid, we can now obtain a maximal representation by first projecting to an $r$-presentation and then by applying Algorithm 1. However, an explicit construction of arbitrary maximal representations is sorely lacking. One could hope that by constructing a random binary matrix and applying Algorithm 1, a random maximal representation could be obtained- but as a heuristic against this, we note that for large $n,r$ there are many matrices that yield the all $1$s matrix but perhaps fewer for other maximal representations. Alternatively, one could determine whether a random matrix was maximal and throw it out from the sampling, but this encounters an issue with equivalence classes due to permutations.

Another barrier to these approaches is complexity. We have the following useful lemma from \cite{Thesis_Rank}.

\begin{lem}
Determining the (generalized) tropical rank of a matrix $X$ is an NP-Hard problem.
\end{lem}

\subsection{Determining matroid rank from a presentation is NP-hard}

With that in mind, we prove the following about viability and maximality.

\begin{thm}
Determining the viability of a representation $X$ of dimension $n \times r$ is an NP-Hard problem while determining maximality is of exponential time complexity.
\end{thm}

\begin{proof}
Toward proving this, we show that viability is equivalent to the tropical rank problem. Suppose that we had an oracle for computing the tropical rank of a matrix. Then given a binary representation $X$, its viability can be computed by the oracle- as the tropical rank. Meanwhile toward maximality, construct $X(c_j)$ for each $j$ indexing the columns. Let $T_j = trank(X(c_j))$. Then for our $X(c_j)$, take the submatrix $M_I$ of $X(c_j)$ where $I$ is a subset of the rows of size $T_j$. Initialize a counter at $0$. For each $M_I$ with tropical rank equal to $T_j$, add one to the counter. Then condition 2 is satisfied if the counter is strictly greater than $T_j$. If for any $j$ this is false, then $X$ is not maximal. If for all $j$ this is true, then $X$ is maximal. In terms of operations, we have two nested loops present- one over the $r$ columns for the construction of $X(c_j)$ and one over ${[n] \choose r}$ in forming $M_I$. So we have that our algorithm for determining maximality using our tropical oracle is
\[
T(n,r) + (r \cdot p_1(n,r)) \cdot \left( T(n,r)+
{[n] \choose r} \cdot p_2(n,r) \cdot T(n,r)
\right)
\]
where $T$ is the operational time of computing the tropical rank of a matrix of size at most $n \times r$, $p_1$ is the construction of $X(c_j)$, and $p_2$ is the construction of $M_I$. Thus, given an oracle for computing the tropical rank, maximality is still an exponential time complexity problem. For fixed rank, however, one can bound it by the usual bound on the choose function $O({n \choose k}) = O(n^k)$.

Conversely, given an oracle to determine viability, we can compute the tropical rank of a matrix. A naive approach suggests checking the viability of all submatrices of $X$ to determine the rank, but doing so yields a complexity in $O({n \choose k})$. Instead, first check whether $X$ is viable or not. If it is, then it has tropical rank $r$ and we are done. If not, append a row of all $1$s to $X$, call it $X^{(1)}$ This increases the tropical rank of a rank-deficient matrix by $1$. Check the viability of $X^{(1)}$ as an $(n+1) \times r$ matrix. Repeat this procedure until you obtain $X^{(i)}$, a viable $(n+i) \times r$ matrix. The tropical rank of $X$ is $r-i$.
\end{proof}

Therefore, determining viability is equivalent to the tropical determinant while maximality seems to be even worse. This implies that determining the rank of a transversal matroid is an NP-hard problem.

\bibliography{manuscript.bbl}{}

\end{document}